\documentclass{amsart}
\usepackage{amsmath,amsxtra,amsthm,amssymb,mathtools,stmaryrd}
\usepackage[b]{esvect}
\usepackage{tikz-cd}
\usepackage{tcolorbox}
\usepackage{hyperref}

\hypersetup{colorlinks=true, linkcolor=blue}

\newtheorem{thm}{Theorem}[section]
\newtheorem{prop}[thm]{Proposition}
\newtheorem{lem}[thm]{Lemma}
\newtheorem{rem}[thm]{Remark}

\DeclareMathOperator{\Hom}{Hom}
\DeclareMathOperator{\End}{End}
\DeclareMathOperator{\Ext}{Ext}

\DeclareMathOperator{\modcat}{mod}
\DeclareMathOperator{\Mod}{Mod}
\DeclareMathOperator{\Der}{Der}

\DeclareMathOperator{\tr}{tr}
\DeclareMathOperator{\Tr}{Tr}
\DeclareMathOperator{\Ev}{\mathbb E}
\DeclareMathOperator{\id}{id}
\DeclareMathOperator{\rad}{rad}

\newcommand{\iso}{\xrightarrow{\raisebox{-.5ex}[0ex][0ex]{$\scriptstyle{\sim}$}}}
\newcommand{\leftiso}{\xleftarrow{\raisebox{-.5ex}[0ex][0ex]{$\scriptstyle{\sim}$}}}

\newcommand*{\intref}[2]{\hyperref[#2]{#1~\ref*{#2}}}

\newcommand{\rmod}[1]{\kern-0.3ex\vv{#1}\kern-0.3ex}
\newcommand{\rmods}[1]{\vv{#1}}
\newcommand{\lmod}[1]{\reflectbox{$\rmod{\reflectbox{$#1$}}$}}
\newcommand{\lmods}[1]{\reflectbox{$\rmods{\reflectbox{$#1$}}$}}

\newcommand{\rlin}[1]{\kern0.3ex\textrm{-}#1}
\newcommand{\llin}[1]{#1\textrm{-}}

\title[The invariance of the AR Formula]{The invariance of the Auslander--Reiten Formula for hereditary algebras}
\author{Andrew Hubery}
\address{Bielefeld University\\33501 Bielefeld\\Germany}
\email{hubery@math.uni-bielefeld.de}
\subjclass{MSC-class: 16D90, 16E30 (Primary), 16G20, 16G70 (Secondary)}

\begin{document}

\begin{abstract}
We show that the Auslander--Reiten Formula for a finite dimensional hereditary algebra is invariant under the Auslander--Reiten translate.
\end{abstract}

\maketitle

\section{Introduction}

Let $\Lambda$ be a finite dimensional hereditary algebra over a field $k$, with Auslander--Reiten translate $\tau$ and inverse translate $\tau^-$. The Auslander--Reiten Formula says that there is a natural isomorphism of bifunctors
\[ \Ext^1(X,Y) \cong D\Hom(\tau^-Y,X) \]
for all finite dimsional modules $X,Y\in\modcat\Lambda$. Alternatively, we can express this as a bifunctorial perfect pairing
\[ \{-,-\}' \colon \Ext^1(X,Y)\times \Hom(\tau^-Y,X) \to k. \]
If now $X$ has no non-zero injective direct summands, then $\tau^-\zeta$ is again exact for all $\zeta\in\Ext^1(X,Y)$, so given $f\in\Hom(\tau^-Y,X)$, we can compute both $\{\zeta,f\}'$ and $\{\tau^-\zeta,\tau^-f\}'$. Our main result is the these two expressions always agree. Beyond settling this natural question, this result plays a key role in showing that, over the associated preprojective algebra $\Pi$, we have a natural isomorphism of bifunctors $\Ext^2_\Pi(X,Y)\cong D\Hom(Y,X)$ \cite{Hubery}.

We remark that our result holds for all finite dimensional hereditary algebras, and not just for those which are tensor algebras. This phenomenon can occur whenever the semisimple algebra $\Lambda/J(\Lambda)$ is not separable over $k$, which happens even for some tame hereditary algebras. Of course, if one restricts to tensor algebras, or even further to path algebras of quivers, then certain constructions admit canonical splittings which can be exploited to simplify the proof.

\subsection*{Acknowledgements}

This work was supported by Deutsche Forschungsgemeinschaft (Project-ID 491392403 -- TRR 358).

\section{Finite dimensional hereditary algebras}

We fix a base field $k$. Let $\Lambda$ be a finite dimensional hereditary algebra. It is known that we can always write $\Lambda=A\oplus J$ such that $A$ is a semisimple algebra and $J$ is the Jacobson radical \cite{Chase,Zaks2}. It follows that the epimorphism $J\to J/J^2$ splits as left $A$-modules, and also as right $A$-modules, but in general it will not be split as $A$-bimodules. In fact, this happens if and only if $\Lambda$ is isomorphic to the tensor algebra $T_A(J/J^2)$.

\begin{rem}
The natural map $J\to J/J^2$ splits as $A$-bimodules in the following situations.
\begin{enumerate}
\item The ext-quiver of $\Lambda$ is a tree. This includes all hereditary algebras of finite representation type. See \cite[Proposition 10.2]{DR1}.
\item There are no indecomposable projective modules $P,P'$ with $\rad^n(P,P')\neq0$ for $n=1,r$ for some $r>1$. See \cite[Section 5]{DR3}.
\item The $k$-algebra $A$ is separable. See \cite{Wedderburn}.
\end{enumerate}
On the other hand, there are finite dimensional hereditary algebras which are not tensor algebras. See \cite{Zaks1,DR3}, where it is shown that this fails even for tame hereditary algebras.
\end{rem}

Associated to $\Lambda$ is the fundamental short exact sequence
\[ \begin{tikzcd}
\mathbb P \ \colon &[-20pt]
0 \arrow[r] & \Omega \arrow[r,"i"] & \Lambda\otimes_A\Lambda \arrow[r,"p"] & \Lambda \arrow[r] & 0
\end{tikzcd} \]
where $p$ is the usual multiplication map. Note that this splits as left $\Lambda$-modules, and as right $\Lambda$-modules.

The kernel $\Omega$ is sometimes called the bimodule of noncommutative 1-forms on $\Lambda$ over $A$. Just as in the construction of the bar resolution we can identify $\Omega$ with the image of the bimodule homomorphism
\[ \Lambda\otimes_A\Lambda\otimes_A\Lambda \to \Lambda\otimes_A\Lambda, \quad
\lambda\otimes_A\mu\otimes_A\nu \mapsto \lambda\otimes_A\mu\nu - \lambda\mu\otimes_A\nu. \]
We write $d\mu$ for the image of $1\otimes_A\mu\otimes_A1$, so that $i(d\mu)=1\otimes_A\mu-\mu\otimes_A1$.

We observe that if $B$ is any $\Lambda$-bimodule, then $\Hom_{\Lambda\textrm{-}\Lambda}(\Omega,B)\cong\Der_A(\Lambda,B)$ is the space of all $A$-derivations $\Lambda\to B$. Explicitly these are the $A$-bimodule homomorphisms $d\colon\Lambda\to B$ satsifying the Leibniz rule $d(\lambda\mu)=\lambda d(\mu)-d(\lambda)\mu$, or equivalently those $k$-derivations $d\colon\Lambda\to B$ satisfying $d(a)=0$ for all $a\in A$ (cf. \cite[Section 10]{Ginzburg}, though the relative case is not considered there).

In the special situation that $\Lambda=T_A(M)$ is a tensor algebra, then $\Omega\cong\Lambda\otimes_AM\otimes_A\Lambda$, and $\Der_A(\Lambda,X)\cong\Hom_{A\textrm{-}A}(M,X)$. 

\begin{prop}\label{prop:standard-proj-inj}
Every right $\Lambda$-module $X$ admits a standard projective resolution
\[ \begin{tikzcd}
\mathbb P_X \ \colon &[-20pt]
0 \arrow[r] & X\otimes\Omega \arrow[r,"i_X"] & X\otimes_A\Lambda \arrow[r,"p_X"] & X \arrow[r] & 0
\end{tikzcd} \]
where $i_X(x\otimes d\lambda)=x\otimes_A\lambda-x\lambda\otimes_A1$ and $p_X(x\otimes_A\lambda)=x\lambda$, as well as a standard injective coresolution
\[ \begin{tikzcd}
\mathbb I_X \ \colon &[-20pt]
0 \arrow[r] & X \arrow[r,"j_X"] & \Hom_A(\Lambda,X) \arrow[r,"q_X"] & \Hom(\Omega,X) \arrow[r] & 0
\end{tikzcd} \]
where $j_X(x)(\lambda)=x\lambda$ and $q_X(h)(d\lambda)=h(\lambda)-h(1)\lambda$.
\end{prop}

Here, and elsewhere except in \intref{Section}{sec:semisimple}, unadorned tensor products, and homomorphism and extension spaces, will be over $\Lambda$.

\begin{proof}
We have $\mathbb P_X=X\otimes\mathbb P$, and this is exact as $\mathbb P$ splits as a sequence of left $\Lambda$-modules. As $A$ is semisimple the middle term $X\otimes_A\Lambda$ is a projective right $\Lambda$-module, and hence so too is $X\otimes\Omega$ since $\Lambda$ is hereditary.

Similarly, we have the short exact sequence $\Hom(\mathbb P,X)$, and up to sign this equals $\mathbb I_X$. Again, the middle term is an injective right $\Lambda$-module, and hence so too is $\Hom(\Omega,X)$.
\end{proof}

\begin{rem}
We note that, regarding $\mathbb P$ as a chain complex in degrees $\{-2,-1,0\}$, and the module $X$ as a stalk complex in degree 0, the sign conventions from \cite[Section 0FNG]{Stacks} tell us that $\Hom(\mathbb P,X)$ is a chain complex in degrees $\{0,1,2\}$ having the differentials $-\Hom(p,X)$ and $\Hom(i,X)$. We have chosen instead to take $j_X=\Hom(p,X)$ and $q_X=-\Hom(i,X)$. The important point, though, is that there is some sign involved, so we should not take both $\Hom(p,X)$ and $\Hom(i,X)$. This viewpoint is further justified \intref{Lemma}{lem:proj-inj-comparison}.
\end{rem}

\begin{rem}
The dual result for left $\Lambda$-modules also holds.
\end{rem}

\begin{rem}
For a projective module $P$ we know that both $P$ and $(P/PJ)\otimes_A\Lambda$ are projective covers of $P/PJ$, so are isomorphic. In particular, we can apply this to $X\otimes\Omega$ to deduce that $X\otimes\Omega\cong X\otimes_AM\otimes_A\Lambda$, where $M\coloneqq J/J^2$.

For, considering the standard projective resolution $\mathbb P\otimes A$ gives $\Omega\otimes A\cong J$ as $\Lambda$-$A$-bimodules, so $X\otimes\Omega\otimes A\cong X\otimes J$ as right $A$-modules. Also, as $\Lambda$ is hereditary, $J^i$ is projective, and hence $X\otimes J^i\cong XJ^i$. Thus, as right $A$-modules, we have $XJ\cong\bigoplus_{i\geq1}XJ^i/XJ^{i+1}$.

Similarly, we have
\begin{multline*}
(XJ^{i-1}/XJ^i)\otimes_A(J/J^2) \cong (XJ^{i-1}/XJ^i)\otimes(J/J^2)\\
\cong (XJ^{i-1}\otimes J)/(XJ^i\otimes J+XJ^{i-1}\otimes J^2) \cong XJ^i/XJ^{i+1}.
\end{multline*}
Thus
\[ X\otimes_AM \cong \bigoplus_{i\geq1}(XJ^{i-1}/XJ^i)\otimes_A(J/J^2) \cong \bigoplus_{i\geq1}XJ^i/XJ^{i+1}, \]
proving the claim.

In general, however, the isomorphism $X\otimes\Omega\cong X\otimes_AM\otimes_A\Lambda$ is not natural in $X$. In fact, if it is natural in $X$, then taking $X=A$ yields an isomorphism of $A$-$\Lambda$-bimodules $J\cong A\otimes\Omega\cong M\otimes_A\Lambda$. The canonical map $J\to J/J^2$ thus becomes $M\otimes_A\Lambda\to M$, which is clearly split as $A$-bimodules, and hence $\Lambda=T_A(M)$ is a tensor algebra. The converse also holds.
\end{rem}

\section{Duality for semisimple algebras}\label{sec:semisimple}

We recall some basic results concerning duality for semisimple algebras.

A finite dimensional algebra $A$ is symmetric provided it is isomorphic, as $A$-bimodules, to its dual $DA=\Hom_k(A,k)$. By considering the image of the identity it is clear that $A$-bimodule homomorphisms $A\to DA$ are in bijection with elements $\sigma\in DA$ satisfying $\sigma(ab)=\sigma(ba)$, and we have an isomorphism if and only if $\sigma(a-)=0$ implies $a=0$. Such $\sigma\in DA$ are called symmetrising elements.

For convenience we prove the following standard lemma.

\begin{lem}
Every finite dimensional semisimple algebra is symmetric.
\end{lem}

\begin{proof}
If $A$  and $A'$ have symmetrising elements $\sigma$ and $\sigma'$ respectively, then the matrix ring $M_n(A)$ has symmetrising element $(a_{ij})\mapsto\sum_i\sigma(a_{ii})$, and the product $A\times A'$ has symmetrising element $(a,a')\mapsto\sigma(a)+\sigma'(a')$. By the Artin--Wedderburn Theorem is is thus enough to show that every finite dimensional division algebra $B$ is symmetric.

Let $K$ be the centre of $B$, and $L/K$ a splitting field for $B/K$. Then $L\otimes_KB\cong M_n(L)$ is a matrix ring, and writing $C(B)$ for the $K$-subspace spanned by all commutators $[a,b]=ab-ba$ we clearly have $L\otimes_KC(B)\subseteq C(L\otimes_KB)$. The latter is a proper subspace of $M_n(L)$, since not all matrices have trace zero. Now every nonzero $\sigma\in DB$ vanishing identically on $C(B)$ is a symmetrising element.
\end{proof}

For the remainder of this section we will fix a semisimple algebra $A$ with symmetrising element $\sigma$. Moreover, all unadorned homomorphism spaces and tensor products will be over $A$.

\begin{lem}
Let $X$ and $Y$ be right $A$-modules with $X$ finite dimensional. Then there are isomorphisms
\[ \Hom(X,Y) \leftiso Y\otimes\Hom(X,A) \iso Y\otimes DX \]
sending $y\otimes f$ in the middle to the map $x\mapsto yf(x)$ on the left, and the element $y\otimes\sigma f$ on the right.
\end{lem}

\begin{proof}
The isomorphism on the left follows since $X$ is a finitely generated projective $A$-module. The map on the right is induced by the isomorphism $A\iso DA$.
\end{proof}

We write $\xi\mapsto\rmod\xi$ for the inverse of the isomorphism $\Hom(X,A)\iso DX$, $f\mapsto\sigma f$. Thus $\rmod\xi(xa)=\rmod\xi(x)a$ and $\sigma(\rmod\xi(x))=\xi(x)$. Note also that this is an isomorphic of left $A$-modules, so $\rmods{a\xi}=a\cdot\rmod\xi$, which is the map $x\mapsto a\rmod\xi(x)$.

Taking $X=Y$ in the lemma yields the special case $\End(X)\cong X\otimes DX$. Let $\sum_ix_i\otimes\xi_i$ correspond to the identity on $X$, so that $x=\sum_ix_i\rmod\xi_i(x)$ for all $x\in X$. We can then write the isomorphism from the lemma as
\[ \Hom(X,Y) \iso Y\otimes DX, \quad f \mapsto \sum_if(x_i)\otimes\xi_i. \]
For, we just need to check that $f=\sum_if(x_i)\rmod\xi_i$, which follows from
\[ \sum f(x_i)\rmod\xi_i(x) = \sum f(x_i\rmod\xi_i(x)) = f(x). \]
Combining with the standard isomorphism $Y\otimes DX\cong D\Hom(Y,X)$ we see that $\Hom(X,Y)$ and $\Hom(Y,X)$ are naturally dual to one another. Explicitly, this sends the pair $f\colon X\to Y$ and $g\colon Y\to X$ to $\sum_i\xi_i(gf(x_i))$.

\begin{rem}
\begin{enumerate}
\item There is an analogous result for left modules. In this case we write the isomorphism as
\[ DX \iso \Hom(X,A), \quad \xi\mapsto\lmod\xi, \]
and for an arbitrary left module $Y$ we have $\Hom(X,Y)\cong DX\otimes Y$.
\item Given a finite dimensional right module $X$, and a left module $Y$, we have the isomorphism
\[ X\otimes Y \cong \Hom(DX,Y), \quad x\otimes y\mapsto\big(\xi\mapsto\rmod\xi(x)y\big). \]
\end{enumerate}
\end{rem}

\section{The Auslander--Reiten Formula}

The Auslander--Reiten translate $\tau=D\Tr$ is the endofunctor of $\modcat\Lambda$ given by
\[ \tau X \coloneqq D\Ext^1(X,\Lambda), \]
and there is a natural isomorphism of bifunctors
\[ \Ext^1(X,Y) \cong D\Hom(Y,\tau X). \]
This we can alternatively express as a bifunctorial perfect pairing
\[ \{-,-\} \colon \Hom(Y,\tau X)\times\Ext^1(X,Y) \to k. \]
Note that the bifunctoriality implies that
\[ \{f,y\zeta\} = \{fy,\zeta\} \quad\textrm{for $\zeta\in\Ext^1(X,Y)$, $f\colon Y'\to\tau X$, $y\colon Y\to Y'$} \]
and
\[ \{g,\zeta x\} = \{\tau(x)g,\zeta\} \quad\textrm{for $\zeta\in\Ext^1(X,Y)$, $g\colon Y\to\tau X'$, $x\colon X'\to X$}. \]

The functor $\tau$ admits a left adjoint $\tau^-=\Tr D$, given on $\modcat\Lambda$ by
\[ \tau^-X \coloneqq \Ext^1(DY,\Lambda), \]
and so yields the bifunctorial perfect pairing
\[ \{-,-\}' \colon \Ext^1(X,Y)\times\Hom(\tau^-Y,X) \to k, \quad \{\zeta,f\}' \coloneqq \{\tilde f,\zeta\}. \]
Here and in the sequel, given an adjoint pair of endofunctors $(F,G)$ on $\modcat\Lambda$, we will denote by $\tilde f\in\Hom(Y,GX)$ the morphism corresponding to $f\in\Hom(FY,X)$, so that $\tilde f=Gf\cdot u_Y$ with $u\colon\id\Rightarrow GF$ the unit of the adjunction. 

Together, the isomorphisms
\[ D\Hom(Y,\tau X) \cong \Ext^1(X,Y) \cong D\Hom(\tau^-Y,X) \]
constitute the Auslander--Reiten Formula. Expressed as perfect pairings we then have the following question concerning invariance under the translate.

Given $\zeta\in\Ext^1(X,Y)$ with $\Hom(D\Lambda,X)=0$,we know $\tau^-\zeta\in\Ext^1(\tau^-X,\tau^-Y)$ is again exact. How then, for $f\in\Hom(\tau^-Y,X)$, are $\{\zeta,f\}'$ and $\{\tau^-\zeta,\tau^-f\}'$ related?

\subsection{Rewriting the Auslander--Reiten translate}

In order to tackle this question we will need a more convenient form for the functors $\tau^\pm$, which will have the dual advantages of being left/right symmetric, as well as providing alternative descriptions of the pairings that are more amenable to computations. To this end we define the exact sequence of $\Lambda$-bimodules
\[ \begin{tikzcd}
\mathbb T \ \colon &[-20pt]
\Lambda\otimes_A\Lambda \arrow[r,"\tilde c"] & \mho \arrow[r,"b"] & \Pi_1 \arrow[r] & 0
\end{tikzcd} \]
where
\[ \mho \coloneqq D(D\Lambda\otimes\Omega\otimes D\Lambda) \]
and $c\coloneqq\tilde c(1\otimes_A1)$ is the map
\[ \theta\otimes d\lambda\otimes\phi \mapsto \theta(\lmod\phi(\lambda)) - \phi(\rmod\theta(\lambda)). \]

We then have the two bimodule isomorphims
\[ \Phi_R \colon \mho \iso \Hom_{\rlin\Lambda}(D\Lambda\otimes\Omega,\Lambda) \quad\textrm{and}\quad
\Phi_L \colon \mho \iso \Hom_{\llin\Lambda}(\Omega\otimes D\Lambda,\Lambda). \]
Here we have used the notation $\Hom_{\rlin\Lambda}(-,-)$ to emphasise that we are dealing with homomorphisms of right $\Lambda$-modules, and similarly $\Hom_{\llin\Lambda}(-,-)$ for homomorphisms of left $\Lambda$-modules.

By definition we then have, for all $\theta,\phi\in D\Lambda$,
\[ \phi(\Phi_R(\gamma)(\theta\otimes d\lambda)) = \gamma(\theta\otimes d\lambda\otimes\phi)
= \theta(\Phi_L(\gamma)(d\lambda\otimes\phi)). \]
In particular,
\[ \Phi_R(c)(\theta\otimes d\lambda) = \rmod\theta(1)\lambda-\rmod\theta(\lambda) \quad\textrm{and}\quad
\Phi_L(c)(d\lambda\otimes\phi) = \lmod\phi(\lambda) - \lambda\,\lmod\phi(1). \]
Here we have used, for example, that
\[ \theta(\,\lmod\phi(\lambda)) = \sigma(\rmod\theta(1)\lmod\phi(\lambda)) = \phi(\rmod\theta(1)\lambda). \]

\begin{rem}
Suppose $\Lambda=T_A(M)$ is a tensor algebra, so that $\Omega=\Lambda\otimes_AM\otimes_A\Lambda$ and $\mho=\Lambda\otimes_ADM\otimes_A\Lambda$. Let us write
\[ \id_{M_A} = \sum_i m_i\otimes_A\mu_i \in M\otimes_ADM \quad\textrm{and}\quad
\id_{{}_AM} = \sum_j \nu_j\otimes_An_j \in DM\otimes_AM. \]
Then, under the natural isomorphisms
\begin{multline*}
\Lambda\otimes_ADM\otimes_A\Lambda \cong \Hom_{\rlin A}(M,\Lambda)\otimes_A\Lambda \cong
\Hom_{\rlin A}(D\Lambda,\Hom_{\rlin A}(M,\Lambda))\\
\cong \Hom_{\rlin A}(D\Lambda\otimes_AM,\Lambda) \cong D(D\Lambda\otimes_AM\otimes_AD\Lambda)
\end{multline*}
the element $\sum_i m_i\otimes_A\mu_i\otimes_A1$ in $\mho$ corresponds first to the element $\iota\otimes_A1$, where $\iota\colon M\rightarrowtail\Lambda$ is the inclusion map, then to the morphism $\theta\otimes_Am\mapsto\lmod\theta(1)m$, and finally to the map
\[ D\Lambda\otimes_AM\otimes_AD\Lambda \to k, \quad
\theta\otimes_Am\otimes_A\phi \mapsto \phi(\,\lmod\theta(1)m) = \theta(\,\lmod\phi(m)). \]
Note that $\lmod\theta(1)=\rmod\theta(1)$ for all $\theta\in D\Lambda$, so this agrees with the earlier computation.

Similarly, the element $\sum_j1\otimes_A\nu_j\otimes_An_j$ in $\mho$ corresponds to the map
\[ D\Lambda\otimes_AM\otimes_AD\Lambda \to k, \quad
\theta\otimes_Am\otimes_A\phi \mapsto \theta(m\rmod\phi(1)) = \phi(\rmod\theta(m)). \]
We deduce that
\[ c = \sum_im_i\otimes_A\mu_i\otimes_A1 - \sum_j1\otimes_A\nu_j\otimes_An_j = \id_{M_A} - \id_{{}_AM}. \]
As a special case, assume $\Lambda=kQ$ is the path algebra of an (acyclic) quiver. Then $A=kQ_0\cong k^n$, and  $M=kQ_1=\bigoplus ka$ is the $A$-bimodule of arrows. Thus $DM=\bigoplus ka^\ast$ has basis the dual arrows, and
\[ c = \sum a\otimes_Aa^\ast\otimes_A1 - \sum 1\otimes_Aa^\ast\otimes_Aa = \sum [a,a^\ast]. \]
\end{rem}

For a right $\Lambda$-module $X$ we have a $\Lambda$-bimodule homomorphism
\[ DX\otimes_kX \to D\Lambda, \quad \xi\otimes_kx \mapsto \big([\xi,x] \colon \lambda \mapsto \xi(x\lambda)\big). \]

\begin{lem}
The right $\Lambda$-module $X\otimes\mho$ is projective for all $X\in\Mod\Lambda$. Moreover, for all $X,Y\in\Mod\Lambda$ with $X$ finite dimensional, there is an isomorphism
\[ Y\otimes\Omega\otimes DX \iso \Hom(X\otimes\mho,Y) \]
sending the element $y\otimes d\lambda\otimes\xi$ to the morphism
\[ x\otimes\gamma \mapsto y\Phi_L(\gamma)(d\lambda\otimes[\xi,x]). \]
\end{lem}

\begin{proof}
Suppose first that $X$ is finite dimensional. Using $\Phi_L$ we have
\[ X\otimes\mho \cong X\otimes\Hom_{\llin\Lambda}(\Omega\otimes D\Lambda,\Lambda) \cong
X\otimes\Hom_{\llin\Lambda}(D\Lambda,\Hom_{\llin\Lambda}(\Omega,\Lambda)). \]
We know from \intref{Proposition}{prop:standard-proj-inj} that $\Hom_{\llin\Lambda}(\Omega,\Lambda)$ is an injective left $\Lambda$-module, so the natural homomorphism
\[ X\otimes\Hom_{\llin\Lambda}(D\Lambda,\Hom_{\llin\Lambda}(\Omega,\Lambda)) \to
\Hom(\Hom(X,D\Lambda),\Hom_{\llin\Lambda}(\Omega,\Lambda)) \]
is an isomorphism. The latter is now isomorphic to
\[ \Hom(DX,\Hom_{\llin\Lambda}(\Omega,\Lambda)) \cong \Hom(\Omega\otimes DX,\Lambda). \]
As $\Omega\otimes DX$ is a finite dimensional projective left module, we conclude that $X\otimes\mho$ is a finite dimensional projective right module.

In general, as $\Lambda$ is finite dimensional, a general module $X$ is the directed colimit (union) of its finite dimensional submodules. Thus $X\otimes\mho$ is a directed colimit of projective modules, hence flat. Again, as $\Lambda$ is finite dimensional, it is perfect, and so all flat modules are projective.

We restrict once more to a finite dimensional module $X$. Then
\[ \Hom(X\otimes\mho,\Lambda) \cong \Hom(\Hom(\Omega\otimes DX,\Lambda),\Lambda) \cong \Omega\otimes DX. \]
Now, for an arbitrary module $Y$, the natural homomorphism
\[ Y\otimes\Hom(X\otimes\mho,\Lambda) \to \Hom(X\otimes\mho,Y\otimes\Lambda) \cong \Hom(X\otimes\mho,Y) \]
is an isomorphism, as $X\otimes\mho$ is a finite dimensional projective.

Chasing through the constructions reveals the explicit description of the morphism $Y\otimes\Omega\otimes DX\iso\Hom(X\otimes\mho,Y)$.
\end{proof}

\begin{rem}
Using standard isomorphisms we obtain the same result after exchanging $\Omega$ and $\mho$. For $X,Y\in\modcat\Lambda$ we have
\[ Y\otimes\mho\otimes DX \cong D\Hom(Y\otimes\mho,X) \cong D\Hom(X\otimes\Omega\otimes DY)
\cong \Hom(X\otimes\Omega,Y). \]
\end{rem}

\begin{lem}\label{lem:tau-new}
For all $X,Y\in\modcat\Lambda$ we have the commutative diagram
\[ \begin{tikzcd}
\Hom(Y\otimes\mho,X) \arrow[d,"{\Hom(1\otimes\tilde c,X)}"] &
X\otimes\Omega\otimes DY \arrow[l,"\sim" swap] \arrow[d,"i_X\otimes1"] \arrow[r,"\sim"]
& D\Hom(X\otimes\Omega,Y) \arrow[d,"{D\Hom(i_X,Y)}"]\\
\Hom_A(Y,X) & X\otimes_ADY \arrow[l,"\sim" swap] \arrow[r,"\sim"] & D\Hom_A(X,Y).
\end{tikzcd} \]
\end{lem}

\begin{proof}
Starting from $x\otimes d\lambda\otimes\eta$ in $X\otimes\Omega\otimes DY$, this is sent on the left to the morphism $Y\otimes\mho\to X$, $y\otimes\gamma\mapsto x\Phi_L(\gamma)(d\lambda\otimes[\eta,y])$, so precomposing with $1\otimes\tilde c$ yields the $A$-linear map
\[ y \mapsto x\Phi_L(c)(d\lambda\otimes[\eta,y]) = x\big(\lmods{[\eta,y]}(\lambda)-\lambda\,\lmods{[\eta,y]}(1)\big). \]
Next, for all $a\in A$ we have
\[ \sigma(a\,\lmods{[\eta,y]}(\lambda)) = [\eta,y](a\lambda) = \eta(ya\lambda) = (\lambda\eta)(ya) = \sigma(\rmods{\lambda\eta}(y)a). \]
Thus the non-degeneracy of $\sigma$ yields $\lmods{[\eta,y]}(\lambda)=\rmods{\lambda\eta}(y)$. We can therefore write the above $A$-linear map as $y\mapsto x\big(\rmods{\lambda\eta}(y)-\lambda\,\rmod\eta(y)\big)$.

On the other hand, $x\otimes d\lambda\otimes\eta$ is sent under $i_X\otimes1$ to $x\otimes_A\lambda\eta-x\lambda\otimes_A\eta$, which is then sent to the same $A$-linear map $y\mapsto x\cdot\rmods{\lambda\eta}(y)-x\lambda\cdot\rmod\eta(y)$.

The commutativity of the right hand side of the diagram is clear, using the natural isomorphism $-\otimes DY\cong D\Hom(-,Y)$. 
\end{proof}

\begin{prop}\label{prop:tau-new}
There is a unique isomorphism $\Hom(\Pi_1,X)\cong D\Ext^1(X,\Lambda)$, natural in $X$, and making the following square commute
\[ \begin{tikzcd}
\Hom(\Pi_1,X) \arrow[r,"\sim"] \arrow[d,"{\Hom(b,X)}"] & D\Ext^1(X,\Lambda) \arrow[d]\\
\Hom(\mho,X) \arrow[r,"\sim"] & D\Hom(X\otimes\Omega,\Lambda)
\end{tikzcd} \]
where the right hand vertical morphism is the dual of the pushout of $\mathbb P_X$.
\end{prop}

\begin{proof}
We consider the diagram from the previous lemma in the special case $Y=\Lambda$, and then take kernels of the outer vertical maps.
\end{proof}

\subsection{The evaluation map}

We introduce the evaluation map for $X\in\modcat\Lambda$
\[ \Ev \colon D\End(X) \to k, \quad \psi \mapsto \psi(\id), \]
as well as the usual trace map for $Y\in\modcat A$
\[ \tr \colon \End_A(Y) \cong Y\otimes_ADY \to k, \quad y\otimes_A\eta \mapsto \eta(y). \]
We observe that, if $X$ is a $\Lambda$-module, then the trace map on $X\otimes_ADX$ factors through $X\otimes DX$. Using the isomorphism $X\otimes DX\cong D\End(X)$ sending $x\otimes\xi$ to the map $h\mapsto\xi(h(x))$, we see that $\Ev(x\otimes\xi)=\xi(x)=\tr(x\otimes\xi)$.

The next lemma collates some basic results concerning this evaluation map.

\begin{lem}\label{lem:ev-map}
Take $X,Y,P,Q\in\modcat\Lambda$ with $P,Q$ projective.
\begin{enumerate}
\item Let $f\colon P\to Q$ and $g\colon Q\to D\Hom(P,\Lambda)$. Then
\[ gf \in \Hom(P,D\Hom(P,\Lambda)) \cong D(P\otimes\Hom(P,\Lambda)) \cong D\End(P) \]
and
\[ D\Hom(f,\Lambda)g \in \Hom(Q,D\Hom(Q,\Lambda)) \cong D\End(Q), \]
and $\Ev(gf)=\Ev(D\Hom(f,\Lambda)g)$.
\item Let $g\in\Hom(X\otimes\Omega,Y)$ and $h\in\Hom(Y\otimes\mho,X)$. Then
\[ \tilde hg \colon X\otimes\Omega \to \Hom(\mho,X) \cong X\otimes\Omega\otimes D\Lambda
\cong D\Hom(X\otimes\Omega,\Lambda) \]
and
\[ \tilde gh \colon Y\otimes\mho \to \Hom(\Omega,Y) \cong Y\otimes\mho\otimes D\Lambda
\cong D\Hom(Y\otimes\mho,\Lambda), \]
and $\Ev(\tilde hg)=\Ev(\tilde gh)$.
\item Let $g,h$ be as above. Then, using the counit map $\varepsilon$ we can compute $\Ev(\tilde hg)$ via
\[ \Hom(X\otimes\Omega,\Hom(\mho,X)) \cong \Hom(\mho,X)\otimes\mho\otimes DX \xrightarrow{\varepsilon\otimes1} X\otimes DX \xrightarrow{\tr} k \]
and can compute $\Ev(\tilde gh)$ via
\[ \Hom(Y\otimes\mho,\Hom(\Omega,Y)) \cong \Hom(\Omega,Y)\otimes\Omega\otimes DY \xrightarrow{\varepsilon\otimes1} Y\otimes DY
\xrightarrow{\tr} k. \]
\item Let $f\in\Hom(X\otimes_A\Lambda,X)\cong\End_A(X)$. Then $\Ev(j_Xf)=\tr(f)$.
\end{enumerate}
\end{lem}

\begin{proof}
(1) Let us write $\id_P=\sum_ip_i\otimes\alpha_i$ in $P\otimes\Hom(P,\Lambda)$, and similarly $\id_Q=\sum_jq_j\otimes\beta_j$. Then $f(p_i)=\sum_jq_j\beta_j(f(p_i))$. Also, $\sum_i\beta_j(f(p_i))\alpha_i=\beta_jf$ as maps $P\to\Lambda$, since we can evaluate at some $p\in P$ to get
\[ \sum_i\beta_j(f(p_i))\alpha_i(p) = \sum_i\beta_j(f(p_i\alpha_i(p))) = \beta_j(f(p)). \]
Thus
\begin{multline*}
\Ev(gf) = \sum_i g(f(p_i))(\alpha_i) = \sum_{i,j} g(q_j)(\beta_j(f(p_i))\alpha_i)\\
= \sum_j g(q_j)(\beta_jf) = \sum_j D\Hom(f,\Lambda)(g(q_j))(\beta_j) = \Ev(D\Hom(f,\Lambda)g).
\end{multline*}

(2) It is enough to show the result when $g$ corresponds to a pure tensor $y\otimes\gamma\otimes\xi$ in $Y\otimes\mho\otimes DX$, and similarly $h$ corresponds to a pure tensor $x\otimes d\lambda\otimes\eta$ in $X\otimes\Omega\otimes DY$.

We note that $\tilde h\colon Y\to\Hom(\mho,X)\cong X\otimes\Omega\otimes D\Lambda$ is the map $y\mapsto x\otimes d\lambda\otimes[\eta,y]$, and the latter corresponds, as an element of $D\Hom(X\otimes\Omega,\Lambda)$, to $\psi\mapsto\eta(y\psi(x\otimes d\lambda))$. The composite $\tilde hg$ therefore sends $x'\otimes d\mu$ to the map $\psi\mapsto\eta(g(x'\otimes d\mu)\psi(x\otimes d\lambda))$.

Writing the identity on $X\otimes\Omega$ as $\sum_i x_i\otimes d\lambda_i\otimes\psi_i$ in $X\otimes\Omega\otimes\Hom(X\otimes\Omega,\Lambda)$, we have $x\otimes d\lambda=\sum_ix_i\otimes d\lambda_i\cdot\psi_i(x\otimes d\lambda)$, and hence
\[ \Ev(\tilde hg) = \sum_i \tilde hg(x_i\otimes d\lambda_i)(\psi_i) = \sum_i \eta(g(x_i\otimes d\lambda_i)\psi_i(x\otimes d\lambda))
= \eta(g(x\otimes d\lambda)). \]
If we now substitute in for $g=y\otimes\gamma\otimes\xi$ we get
\[ \Ev(\tilde hg) = \eta(y\Phi_R(\gamma)([\xi,x]\otimes d\lambda)) = \gamma([\xi,x]\otimes d\lambda\otimes[\eta,y]). \]
By symmetry the latter expression also equals $\Ev(\tilde gh)$.

(3) We keep the assumptions on $g$ and $h$ being pure tensors. As above, $\tilde h(y)=x\otimes d\lambda\otimes[\eta,y]$, so under the isomorphism
\[ \Hom(X\otimes\Omega,X\otimes\Omega\otimes D\Lambda) \cong X\otimes\Omega\otimes D\Lambda\otimes\mho\otimes DX, \]
we see that $\tilde hg$ corresponds to the pure tensor $x\otimes d\lambda\otimes[\eta,y]\otimes\gamma\otimes\xi$. Using the counit $\Hom(\mho,X)\otimes\mho\to X$, this is sent to $x\Phi_L(\gamma)(d\lambda\otimes[\eta,y])\otimes\xi$ in $X\otimes DX$, and so maps under the trace to
\[ \xi(x\Phi_L(\gamma)(d\lambda\otimes[\eta,y])) = \gamma([\xi,x]\otimes d\lambda\otimes[\eta,y]) = \Ev(\tilde hg). \]
The other statement is proved similarly.

(4) Let us write $\id_\Lambda\in\End_{\rlin A}(\Lambda)$ as $\sum_i\lambda_i\otimes_A\theta_i$ in $\Lambda\otimes_AD\Lambda$, and similarly $\id_X\in\End_A(X)$ as $\sum_sx_s\otimes_A\xi_s$ in $X\otimes_ADX$. For clarity we write $\bar f\in\End_A(X)$.

The isomorphisms $\Hom_A(\Lambda,X)\cong X\otimes_AD\Lambda\cong D\Hom_A(X,\Lambda)$ send $h\colon\Lambda\to X$ first to $\sum_ih(\lambda_i)\otimes_A\theta_i$, and then to the map $\psi\mapsto\sum_i\theta_i(\psi(h(\lambda_i)))$. In particular, putting $h=j_Xf(x\otimes_A\mu)$, we get $\psi\mapsto\sum_i\theta_i(\psi(\bar f(x)\mu\lambda_i))$. This allows us to consider $j_Xf$ as an element of $D(X\otimes_A\Lambda\otimes_A\Hom_A(X,\Lambda))$.

We also have the isomorphisms $\End(X\otimes_A\Lambda)\cong\Hom_A(X,X\otimes_A\Lambda)\cong X\otimes_A\Lambda\otimes_ADX$, under which the identity is sent to $\sum_sx_s\otimes_A1\otimes_A\xi_s$. Using $\sum_iy_i\rmod\theta_i(1)=1$ gives
\begin{multline*}
\Ev(j_Xf) = \sum_{i,s}\theta_i(\rmod\xi_s(\bar f(x_s)\lambda_i)) = \sum_{i,s}\sigma(\rmod\theta_i(1)\rmod\xi_s(\bar f(x_s)\lambda_i))\\
= \sum_{i,s}\xi_s(\bar f(x_s)\lambda_i\rmod\theta_i(1)) = \sum_s\xi_s(\bar f(x_s)) = \tr(\bar f)
\end{multline*}
as required.
\end{proof}

\subsection{Rewriting the Auslander--Reiten Formula}

We saw in \intref{Proposition}{prop:tau-new} that there is a unique natural isomorphism from $\tau X=D\Ext^1(X,\Lambda)$ to $\Hom(\Pi_1,X)$ which is compatible with the dual of pushout and the map $\Hom(b,X)$. By the uniqueness of left adjoints there is similarly a unique natural isomorphism from $\tau^-X$ to $X\otimes\Pi_1$. We also note that the analogous statements for left modules hold, so for example $\tau^-Y\cong\Pi_1\otimes Y$ for a left module $Y$.

We can now express the Auslander--Reiten Formula as bifunctorial pairings
\[ \{-,-\} \colon \Hom(Y,\Hom(\Pi_1,X))\times\Ext^1(X,Y) \to k \]
and
\[ \{-,-\}' \colon \Ext^1(X,Y)\times\Hom(Y\otimes\Pi_1,X) \to k, \]
and we wish to explicitly compute these using, respectively, the projective resolution $\mathbb P_X$ and the injective coresolution $\mathbb I_Y$.

\begin{prop}\label{prop:AR-one}
The Auslander--Reiten Formula can be expressed as the perfect pairing
\[ \{-,-\} \colon \Hom(Y,\Hom(\Pi_1,X))\times\Ext^1(X,Y) \to k, \quad \{f,\zeta\} \coloneqq \Ev(b^\ast fg), \]
where $\zeta=g\cdot\mathbb P_X$ for some $g\colon X\otimes\Omega\to Y$, and $b^\ast f\colon Y\to\Hom(\mho,X)$.
\end{prop}

\begin{proof}
Writing $\alpha\colon\Hom(\Pi_1,X)\iso D\Ext^1(X,\Lambda)$ for the isomorphism from \intref{Proposition}{prop:tau-new}, the pairing $\{\alpha f,\zeta\}$ is computed by composing $\alpha f$ with the dual of the pushout to obtain $\hat f\colon Y\to D\Hom(X\otimes\Omega,\Lambda)$, and then taking $\Ev(\hat fg)$. Now, $\hat f$ equals $b^\ast f$ followed by the natural isomorphism, so $\Ev(\hat fg)=\Ev(b^\ast fg)$ by \intref{Lemma}{lem:ev-map}.
\end{proof}

\begin{lem}\label{lem:proj-inj-comparison}
For all $g\colon X\otimes\Omega\to Y$ we have $g\cdot\mathbb P_X=\mathbb I_Y\cdot\tilde g$.
\end{lem}

\begin{proof}
This follows from the following exact commutative diagram
\[ \begin{tikzcd}
\mathbb P_X\ \colon &[-20pt] 0 \arrow[r] & X\otimes\Omega \arrow[r,"i_X"] \arrow[d,"g"]
& X\otimes_A\Lambda \arrow[r,"p_X"] \arrow[d,"h"] & X \arrow[r] \arrow[d,"\tilde g"] & 0\\
\mathbb I_Y\ \colon &[-20pt] 0 \arrow[r] & Y \arrow[r,"j_Y"] & \Hom_A(\Lambda,Y) \arrow[r,"q_Y"] &
\Hom(\Omega,Y) \arrow[r] & 0
\end{tikzcd} \]
where
\[ h(x\otimes_A\lambda) \colon \mu \mapsto g(x\otimes d(\lambda\mu)). \]
We remark that this result uses our sign conventions $j_Y=\Hom(p,Y)$ and $q_Y=-\Hom(i,Y)$.
\end{proof}

\begin{prop}\label{prop:AR-two}
For $\zeta\in\Ext^1(X,Y)$ and $f\colon\tau^-Y\to X$ we have
\[ \{-,-\}' \colon \Ext^1(X,Y)\times\Hom(Y\otimes\Pi_1,X) \to k, \quad \{\zeta,f\}' \coloneqq \Ev(\tilde gfb_\ast), \]
where $\zeta=\mathbb I_Y\cdot\tilde g$ for some $\tilde g\colon X\to\Hom(\Omega,Y)$, and $fb_\ast\colon Y\otimes\mho\to X$.
\end{prop}

\begin{proof}
We set $h\coloneqq fb_\ast$, so that $\tilde h=b^\ast\tilde f$ as maps $Y\to\Hom(\mho,X)$. Writing $g\colon X\otimes\Omega\to Y$, the previous lemma tells us that $\zeta=g\cdot\mathbb P_X$, so we can use \intref{Proposition}{prop:AR-one} and \intref{Lemma}{lem:ev-map} (2) to get
\[ \{\zeta,f\}' = \{\tilde f,\zeta\} = \Ev(\tilde hg) = \Ev(\tilde gh) = \Ev(\tilde gfb_\ast). \qedhere \]
\end{proof}

\section{Invariance under the translate}

From now on we redefine $\tau X\coloneqq\Hom(\Pi_1,X)$ and $\tau^-X\coloneqq X\otimes\Pi_1$ and use the reformulations of the Auslander--Reiten Formula given in \intref{Propositions}{prop:AR-one} and \intref{}{prop:AR-two}.

In order to show the $\tau$-invariance of the Auslander--Reiten Formula we need a new way of computing the pairing when it involves some $\zeta\in\Ext^1(\tau^-X,Y)$. We first show how to compute it in terms of the projective presentation
\[ \begin{tikzcd}
\mathbb T_X \ \colon &[-20pt]
X\otimes_A\Lambda \arrow[r,"1\otimes\widetilde c"] & X\otimes\mho \arrow[r,"1\otimes b"] & X\otimes\Pi_1 \arrow[r] & 0
\end{tikzcd} \]

\begin{lem}\label{lem:T_X-pairing}
There is a well-defined pairing
\[ \Hom(Y,X) \times \Ext^1(\tau^-X,Y) \to k, \quad (f,\zeta) \mapsto \Ev(j_Xfg) = \tr(fg), \]
where $\zeta=g\cdot\mathbb T_X$.
\end{lem}

\begin{proof}
We begin by remarking that if $\zeta=g\cdot\mathbb T_X$, then $g$ necessarily factors through the image of $1\otimes\tilde c\colon X\otimes_A\Lambda\to X\otimes\mho$. Also, $\Ev(j_Xfg)=\tr(fg)$ by \intref{Lemma}{lem:ev-map} (4).

To see that the pairing is well-defined we need to show that if $\zeta$ is split, equivalently if we have a factorisation $g=h(1\otimes\tilde c)$, then $\Ev(j_Xfg)=0$. Identifying $\Hom_A(\Lambda,X)\cong D\Hom(X\otimes_A\Lambda,\Lambda)$ we have
\[ \Ev(j_Xfg) = \Ev(j_Xfh(1\otimes\tilde c)) = \Ev(D\Hom(1\otimes\tilde c,\Lambda)j_Xfh). \]
By \intref{Lemma}{lem:tau-new} we know that $D\Hom(\id\otimes\tilde c,\Lambda)=\Hom(i,X)=-q_X$, so $\Ev(j_Xfg)=-\Ev(q_Xj_Xfh)=0$ as required.
\end{proof}

\begin{lem}
Take $\zeta\in\Ext^1(\tau^-X,Y)$ and write this as $g\cdot\mathbb T_X=\zeta=\mathbb I_Y\cdot\tilde h$ for some maps $g\colon X\otimes_A\Lambda\to Y$ and $\tilde h\colon\tau^-X\to\Hom(\Omega,Y)$. Then, under the natural maps
\[ \Hom(X\otimes_A\Lambda,Y) \cong Y\otimes_ADX \to Y\otimes DX \]
and
\[ \Hom(X\otimes\mho,\Hom(\Omega,Y)) \cong \Hom(\Omega,Y)\otimes\Omega\otimes DX \xrightarrow{\varepsilon\otimes1} Y\otimes DX \]
the images of $g$ and $\tilde hb_\ast$ differ by a sign.
\end{lem}

\begin{proof}
By assumption there is an exact commutative diagram of the form
\[ \begin{tikzcd}
& X\otimes_A\Lambda \arrow[r,"1\otimes\tilde c"] \arrow[d,"g"] & X\otimes\mho \arrow[r,"1\otimes b"] \arrow[d,"t"] &
\tau^-X \arrow[r] \arrow[d,"\tilde h"] & 0\\
0 \arrow[r] & Y \arrow[r,"j_Y"] & \Hom_A(\Lambda,Y) \arrow[r,"q_Y"] & \Hom(\Omega,Y) \arrow[r] & 0
\end{tikzcd} \]
Using the isomorphism
\[ \Hom(X\otimes\mho,\Hom_A(\Lambda,Y)) \cong \Hom_A(\Lambda,Y)\otimes\Omega\otimes DX \]
we can express $t$ as
\[ t = \sum_i \alpha_i\otimes d\lambda_i\otimes\xi_i, \quad
x\otimes\gamma \mapsto \sum_i \alpha_i\cdot\Phi_L(\gamma)(d\lambda_i\otimes[\xi_i,x]). \]
In particular, as in the proof of \intref{Lemma}{lem:tau-new}, we have
\[ t(x\otimes c) = \sum_i \alpha_i\big(\rmods{\lambda_i\xi_i}(x)-\lambda\rmods{\xi_i}(x)\big) \in \Hom_A(\Lambda,Y). \]
Using that the $\alpha_i$ are right $A$-linear we obtain the equality in $Y$
\[ g(x\otimes_A1) = j_Yg(x\otimes_A1)(1) = t(x\otimes c)(1) = \sum_i\big(\alpha_i(1)\rmods{\lambda_i\xi_i}(x) - \alpha_i(\lambda_i)\rmods{\xi_i}(x)\big). \]
In other words, as an element of $Y\otimes_ADX$ we have
\[ g = \sum_i \big(\alpha_i(1)\otimes_A\lambda_i\xi_i - \alpha_i(\lambda_i)\otimes_A\xi_i\big), \]
which of course maps to $\sum_i\big(\alpha_i(1)\lambda_i-\alpha_i(\lambda_i)\big)\otimes\xi_i$ in $Y\otimes DX$.

On the other hand,
\[ q_Yt = \sum_i q_Y(\alpha_i)\otimes d\lambda_i\otimes\xi_i, \]
so using the counit we have
\[ (\varepsilon\otimes\id)(q_Yt) = \sum_i q_Y(\alpha_i)(d\lambda_i)\otimes\xi_i
= \sum_i \big(\alpha_i(\lambda_i)-\alpha_i(1)\lambda_i\big)\otimes\xi_i \]
and this is the image of $\tilde hb_\ast$.
\end{proof}

\begin{prop}\label{prop:tau-inv-AR-one}
For $\zeta\in\Ext^1(\tau^-X,Y)$ and $f\in\Hom(Y,X)$ we have
\[ \{\zeta,\tau^-f\}' = -\Ev(j_Xfg) = -\tr(fg) \quad\textrm{where } \zeta=g\cdot\mathbb T_X. \]
\end{prop}

\begin{proof}
We begin by noting that $\{\zeta,\tau^-f\}'=\{f\zeta,\id\}'$ by the bifunctoriality of the Auslander--Reiten pairing. Writing $f\zeta=\mathbb I_X\cdot\tilde h$ for some $\tilde h\colon\tau^-X\to\Hom(\Omega,X)$, we then have $\{f\zeta,\id\}'=\Ev(\tilde hb_\ast)=\tr(\varepsilon(\tilde hb_\ast))$ by \intref{Proposition}{prop:AR-two} and \intref{Lemma}{lem:ev-map} (3).

On the other hand, if $\zeta=g\cdot\mathbb T_X$, then \intref{Lemma}{lem:T_X-pairing} gives $\Ev(jfg)=\tr(fg)$. Finally, $f\zeta=fg\cdot\mathbb T_X$, so the previous lemma tells us that the image of $fg$ in $X\otimes DX$ equals $-\varepsilon(\tilde hb_\ast)$, from which the result follows.
\end{proof}

We next want to compute the Auslander--Reiten pairing using a slightly different projective presentation, coming from the epimorphism $X\otimes_A\mho\twoheadrightarrow X\otimes\Pi_1$.

\begin{lem}
The following is an exact commutative diagram between projective presentations of $\tau^-X$
\[ \begin{tikzcd}
\mathbb T'_X \; \colon &[-30pt]
(X\otimes\Omega\otimes\mho)\!\oplus\!(X\otimes_A\Lambda\otimes_A\Lambda)
\arrow[r,"{(i_X\otimes1,1\otimes_A\tilde c)}"] \arrow[d,"{(0,p_X\otimes_A1)}"] &[25pt]
X\otimes_A\mho \arrow[r,"p_X\otimes b"] \arrow[d,"p_X\otimes1"] &[-2pt] X\otimes\Pi_1 \arrow[r] \arrow[d,equal] &[-6pt] 0\\
\mathbb T_X \; \colon &[-30pt]
X\otimes_A\Lambda \arrow[r,"1\otimes\tilde c"] & X\otimes\mho \arrow[r,"1\otimes b"] & X\otimes\Pi_1 \arrow[r] & 0
\end{tikzcd} \]
\end{lem}

\begin{proof}
We have the exact commutative diagram
\[ \begin{tikzcd}
0 \arrow[r] & X\otimes\Omega\otimes_A\Lambda \arrow[r] \arrow[d] & X\otimes_A\Lambda\otimes_A\Lambda \arrow[r] \arrow[d] &
X\otimes_A\Lambda \arrow[r] \arrow[d] & 0\\
0 \arrow[r] & X\otimes\Omega\otimes\mho \arrow[r] \arrow[d] & X\otimes_A\mho \arrow[r] \arrow[d] & X\otimes\mho \arrow[r] \arrow[d] & 0\\
& X\otimes\Omega\otimes\Pi_1 \arrow[r] \arrow[d] & X\otimes_A\Pi_1 \arrow[r] \arrow[d] & X\otimes\Pi_1 \arrow[r] \arrow[d] & 0\\
&0&0&0
\end{tikzcd} \]
where the rows are $\mathbb P_X\otimes_A\Lambda$, $\mathbb P_X\otimes\mho$ and $\mathbb P_X\otimes\Pi_1$, and the columns are $X\otimes\Omega\otimes\mathbb T$, $X\otimes_A\mathbb T$ and $X\otimes\mathbb T$. The exactness of the sequence $\mathbb T'_X$ follows from a diagram chase, and the commutativity of the diagram is clear.
\end{proof}

We now need an analogue of \intref{Lemma}{lem:T_X-pairing} for $\mathbb T'_X$. For this we define
\[ j'=(j'_1,-j'_2)^t \colon X \to (X\otimes\Omega\otimes\mho\otimes D\Lambda)\oplus(X\otimes_A\Lambda\otimes_AD\Lambda), \]
where $j'_1$ corresponds to the identity on $X\otimes\Omega$ under the isomorphism
\[ \Hom(X,X\otimes\Omega\otimes\mho\otimes D\Lambda) \cong \Hom(X,\Hom(\Omega,X\otimes\Omega)) \cong \End(X\otimes\Omega), \]
and $j'_2$ corresponds to the map $X\to\Hom_A(\Lambda,X\otimes_A\Lambda)$, $j'_2(x)\colon\lambda\mapsto x\lambda\otimes_A1$, under the isomorphism
\[ X\otimes_A\Lambda\otimes_AD\Lambda \cong \Hom_A(\Lambda,X\otimes_A\Lambda). \]

\begin{lem}
There is a well-defined pairing
\[ \Hom(Y,X)\times\Ext^1(\tau^-X,Y) \to k, \quad (f,\zeta) \mapsto \Ev(j'fg), \]
where $\zeta=g\cdot\mathbb T'_X$.
\end{lem}

\begin{proof}
As in the proof of \intref{Lemma}{lem:T_X-pairing} we need to show that if we have a factorisation $g=h(i_X\otimes1,1\otimes_A\tilde c)$, then $\Ev(j'fg)=0$. Now, applying $D\Hom(-,\Lambda)\cong-\otimes D\Lambda$ to the map $(i_X\otimes1,1\otimes_A\tilde c)$ we get from \intref{Lemma}{lem:ev-map} (1) that
\[ \Ev(j'fh(i_X\otimes1,1\otimes_A\tilde c)) = \Ev((i_X\otimes1,1\otimes_A\tilde c\otimes1)j'fh). \]
The composite $(i_X\otimes1)j'_1$ goes from $X$ to $X\otimes_A\mho\otimes D\Lambda\cong\Hom(\Omega,X\otimes_A\Lambda)$, and corresponds to the map $i_X\colon X\otimes\Omega\to X\otimes_A\Lambda$, sending $x\otimes d\lambda$ to $x\otimes_A\lambda-x\lambda\otimes_A1$.

On the other hand $1\otimes_A\tilde c\otimes1$ corresponds to
\[ D\Hom(1\otimes_A\tilde c,\Lambda) \colon D\Hom_A(X\otimes_A\Lambda,\Lambda) \to D\Hom(X\otimes_A\mho,\Lambda), \]
and thus, as in \intref{Lemma}{lem:tau-new}, to $\Hom(i,X\otimes_A\Lambda)=-q_{X\otimes_A\Lambda}$. It follows that $(1\otimes_A\tilde c\otimes1)j'_2$, when viewed as a morphism $X\to\Hom(\Omega,X\otimes_A\Lambda)$, sends $x$ to the map $d\lambda\mapsto x\otimes_A\lambda-x\lambda\otimes_A1$.

We conclude that $(i_X\otimes1)j'_1$ and $(1\otimes_A\tilde c\otimes1)j'_2$ agree as maps $X\otimes\Omega\to X\otimes_A\Lambda$, so that $(i_X\otimes1,1\otimes_A\tilde c\otimes1)j'=0$.
\end{proof}

We can now compute the Auslander--Reiten pairing using the presentation $\mathbb T'_X$.

\begin{prop}\label{prop:tau-inv-AR-two}
For $\zeta\in\Ext^1(\tau^-X,Y)$ and $f\in\Hom(Y,X)$ we have
\[ \{\zeta,\tau^-f\}' = \Ev(j'fg) \quad\textrm{where } \zeta=g\cdot\mathbb T'_X. \]
\end{prop}

\begin{proof}
Take $h\colon X\otimes_A\Lambda\to X$ such that $f\zeta=h\cdot\mathbb T_X$. Then also $f\zeta=h(0,p_X\otimes_A1)\cdot\mathbb T'_X$. Using \intref{Proposition}{prop:tau-inv-AR-one} it is enough to show that $\Ev(j'h(0,p_X\otimes_A1))=-\tr(h)$. Applying $D\Hom(-,\Lambda)\cong-\otimes D\Lambda$ to $p_X\otimes_A1$ gives the map $X\otimes_A\Lambda\otimes_AD\Lambda\to X\otimes_AD\Lambda$, or equivalently the map $\Hom_A(\Lambda,p_X)\colon\Hom_A(\Lambda,X\otimes_A\Lambda)\to\Hom_A(\Lambda,X)$. Now $\Hom_A(\Lambda,p_X)j'_2=j_X$, so by \intref{Lemma}{lem:ev-map}, $\Ev(j'h(0,p_X\otimes_A1))=-\Ev(j_Xh)=-\tr(h)$ as required.
\end{proof}

Finally, we can prove the $\tau$-invariance of the Auslander--Reiten pairing.

\begin{thm}
Assume $\Hom(D\Lambda,X)=0$. Then, given $\zeta\in\Ext^1(X,Y)$ and $f\colon\tau^-Y\to X$, we have
\[ \{\zeta,f\}' = \{\tau^-\zeta,\tau^-f\}'. \]
\end{thm}

\begin{proof}
We begin by writing $\zeta=g\cdot\mathbb P_X$ for some $g\colon X\otimes\Omega\to Y$. We also set $h\coloneqq fb_\ast\in\Hom(Y\otimes\mho,X)$. Then $\{\zeta,f\}'=\Ev(\tilde gh)$ by \intref{Proposition}{prop:AR-two}.

Next, $\tau^-\zeta=\tau^-g\cdot\tau^-\mathbb P_X$, which we can compare with $\mathbb T'_X$ via the exact commutative diagram
\[ \begin{tikzcd}
\mathbb T'_X \colon &[-38pt]
(X\otimes\Omega\otimes\mho)\!\oplus\!(X\otimes_A\Lambda\otimes_A\Lambda)
\arrow[r,"{(i_X\otimes1,1\otimes_A\tilde c)}"] \arrow[d,"{(1\otimes b,0)}"] &[25pt]
X\otimes_A\mho \arrow[r,"p_X\otimes b"] \arrow[d,"1\otimes_Ab"] &[-3pt] X\otimes\Pi_1 \arrow[r] \arrow[d,equal] &[-7pt] 0\\
\tau^-\mathbb P_X \colon &[-38pt]
X\otimes\Omega\otimes\Pi_1 \arrow[r,"i_X\otimes1"] & X\otimes_A\Pi_1 \arrow[r,"p_X\otimes1"] & X\otimes\Pi_1 \arrow[r] & 0
\end{tikzcd} \]
so that $\tau^-\zeta=(g\otimes b,0)\cdot\mathbb T'_X$, and thus $\{\tau^-\zeta,\tau^-f\}'=\Ev(j'f(g\otimes b,0))$ by \intref{Proposition}{prop:tau-inv-AR-two}. We now write $(g\otimes b,0)=b_\ast(g\otimes1,0)$, where $g\otimes1\colon X\otimes\Omega\otimes\mho\to Y\otimes\mho$. Then, under the isomorphism
\[ D\Hom(X\otimes\Omega\otimes\mho,\Lambda) \cong X\otimes\Omega\otimes\mho\otimes D\Lambda
\cong \Hom(\Omega,X\otimes\Omega), \]
we have $D\Hom(g\otimes1,\Lambda)=\Hom(\Omega,g)$, and so $\Hom(\Omega,g)j'_1=\tilde g$ as maps $X\to\Hom(\Omega,Y)$. Using \intref{Lemma}{lem:ev-map} (1) we get $\{\tau^-\zeta,\tau^-f\}'=\Ev(\tilde gfb_\ast)=\Ev(\tilde gh)$.
\end{proof}


\begin{thebibliography}{99}
\bibitem{Chase} Chase, S. U., `A generalization of the ring of triangular matrices', Nagoya Math. J. 18 (1961), 13--25.

\bibitem{DR1} Dlab, V., and Ringel, C.M., `On algebras of finite representation type', J. Algebra 33 (1975), 306--394.

\bibitem{DR3} Dlab, V., and Ringel, C.M., `The representations of tame hereditary algebras', in: Representation theory of algebras (Proc. Conf., Temple Univ., Philadelphia, Pa., 1976), 329--353, Lect. Notes Pure Appl. Math. 37 (Dekker, New York, 1978).

\bibitem{Ginzburg} Ginzburg, V., Lectures on Noncommutative Geometry, \url{https://arxiv.org/abs/math/0506603}, 2005.

\bibitem{Hubery} Hubery, A., `On the global dimension and Koszul property for preprojective algebras', \url{https://arxiv.org/abs/2509.21448}, 2025.

\bibitem{Stacks} The Stacks Project Authors, `Stacks Project', \url{https://stacks.math.columbia.edu}, 2026.

\bibitem{Wedderburn} Wedderburn, J. H. M., `On hypercomplex numbers', Proc. London Math. Soc. (2) 6 (1908), 77--118.

\bibitem{Zaks1} Zaks, A., `A note on semi-primary hereditary rings', Pacific J. Math. 23 (1967), 627--628.

\bibitem{Zaks2} Zaks, A., `Semi-primary hereditary rings', Israel J. Math. 6 (1968), 359--362.
\end{thebibliography}
\end{document}